\documentclass[11pt, a4paper]{amsart}

\usepackage[euler-digits]{eulervm}

\usepackage{amsfonts, amsthm, amssymb, amsmath, stmaryrd}
\usepackage{mathrsfs,array}
\usepackage{eucal,fullpage,times,color,enumerate,accents}
\usepackage{url}

\usepackage{color}
\usepackage{mathrsfs}
\usepackage{amssymb}
\usepackage{tikz}
\usepackage{tikz-cd}
\usepackage{bm}
\usepackage{enumerate}

\title{Toric graph associahedra and compactifications of $M_{0,n}$}
\author{Rodrigo Ferreira da Rosa, David Jensen, and Dhruv Ranganathan}
\date{\today}

\address{Rodrigo Ferreira da Rosa \\ Department of Mathematics, Yale University}
\email{rodrigo.ferreiradarosa@yale.edu}

\address{David Jensen \\ Department of Mathematics, University of Kentucky}
\email{dave.jensen@uky.edu}

\address{Dhruv Ranganathan \\ Department of Mathematics, Yale University}
\email{dhruv.ranganathan@yale.edu}

\newtheorem*{theorem*}{Theorem}
\newtheorem{theorem}{Theorem}
\newtheorem{corollary}[theorem]{Corollary}

\newtheorem{proposition}[theorem]{Proposition}
\newtheorem{definition}[theorem]{Definition}

\newtheorem{construction}{Construction}

\newtheorem{quasi-theorem}[theorem]{Quasi-Theorem}

\newtheorem{rem1}[theorem]{Remark}
\newenvironment{remark}{\begin{rem1}\em}{\end{rem1}}

\newtheorem{not1}[theorem]{Notation}

\newcommand{\PP}{\mathbb{P}}
		
\newcommand{\RR} {{\mathbb R}}

\newcommand{\PG}{{\mathcal P}G}

\newcommand{\cal}{\mathcal}

\def\cC{{\cal C}}

\begin{document}

\pagestyle{plain}
\maketitle

\begin{abstract}
To any graph $G$ one can associate a toric variety $X(\mathcal{P}G)$, obtained as a blowup of projective space along coordinate subspaces corresponding to connected subgraphs of $G$. The polytope of this toric variety is the graph associahedron of $G$, a class of polytopes that includes the permutohedron, associahedron, and stellahedron. We show that the space $X(\mathcal{P}{G})$ is isomorphic to a Hassett compactification of $M_{0,n}$ precisely when $G$ is an iterated cone over a discrete set. This may be viewed as a generalization of the well-known fact that the Losev--Manin moduli space is isomorphic to the toric variety associated to the permutohedron.
\end{abstract}

\section{Introduction}

In this note, we study the relationship between two families of blowups of projective space. The first is given by Hassett's spaces of weighted pointed stable rational curves~\cite{Has03}. The second is a family of toric varieties built as blowups of projective space, from polytopes known as \textit{graph associahedra}.

Given a simple graph $G$ on $d+1$ vertices, the graph associahedron $\PG$ is a $d$-dimensional convex polytope, constructed by truncating the $d$-simplex based on the connected subgraphs of $G$. We review this construction in Section~\ref{sec: background}. If $G$ is the complete graph on $d+1$ vertices, then $\PG$ is the permutohedron. Our motivation for the present work goes back to the following beautiful result of Losev and Manin~\cite{LM}.

\begin{theorem*}
Let $X(\mathcal P K_{n-2})$ be the toric variety associated to the $(n-3)$ dimensional permutohedron. There is an isomorphism
\[
X(\mathcal P K_{n-2}) \cong \overline M_{0,n}^{LM},
\]
where $\overline M_{0,n}^{LM}$ is the Losev--Manin space of chains of rational curves with $n$ marked points.
\end{theorem*}

\subsection{Main results} In addition to the permutohedron, graph associahedra contain several important families of polytopes, including the associahedron (or \textit{Stasheff} polytope), the cyclohedron (or \textit{Bott--Taubes} polytope), and the stellahedron. The main result of this paper is a complete classification of graphs $G$ such that $X(\PG)$ is isomorphic to a Hassett compactification of $M_{0,n}$. Recall that given a graph $G$, the cone, denoted $Cone(G)$, is the graph obtained by introducing one new vertex $v_0$ to $G$, and connecting each of the vertices in $G$ to $v_0$. We denote the $\ell$-times iterated cone by $Cone^{\ell}(G)$.

\begin{theorem}
\label{thm: mainthm}
Let $G$ be a graph on $(n-2)$ vertices. Then there exists a weight vector $\omega\in (0,1]^n$ such that
\[
X(\PG)\cong \overline M_{0,\omega}
\]
if and only if $G$ is an iterated cone over a discrete set. That is,
\[
G\cong Cone^{n-k-2}(\sqcup_{i=1}^k v_i).
\]
\end{theorem}

We give a precise description of the weights for the associated moduli space in Remark~\ref{rem: specific-weights}.

As an example, we have the following result for the stellahedron.
\begin{corollary}
If $G$ is a star graph on $n-2$ vertices, then
\[
X(\PG) \cong \overline M_{0,\omega},
\]
where $\omega = (1,\frac{1}{2},\frac{1}{2}+\epsilon,\epsilon,\ldots, \epsilon)$, with $\epsilon<\frac{1}{n}$.
\end{corollary}

Several well studied polytopes do not give rise to Hassett spaces.
\begin{corollary}
Let $|V(G)|\geq 4$. If $G$ is a path graph, a cycle, or a complete bipartite or multipartite graph, then $X(\PG)$ is not isomorphic to $\overline M_{0,\omega}$ for any choice of weight vector $\omega$.
\end{corollary}

\begin{figure}
\includegraphics[scale=0.3]{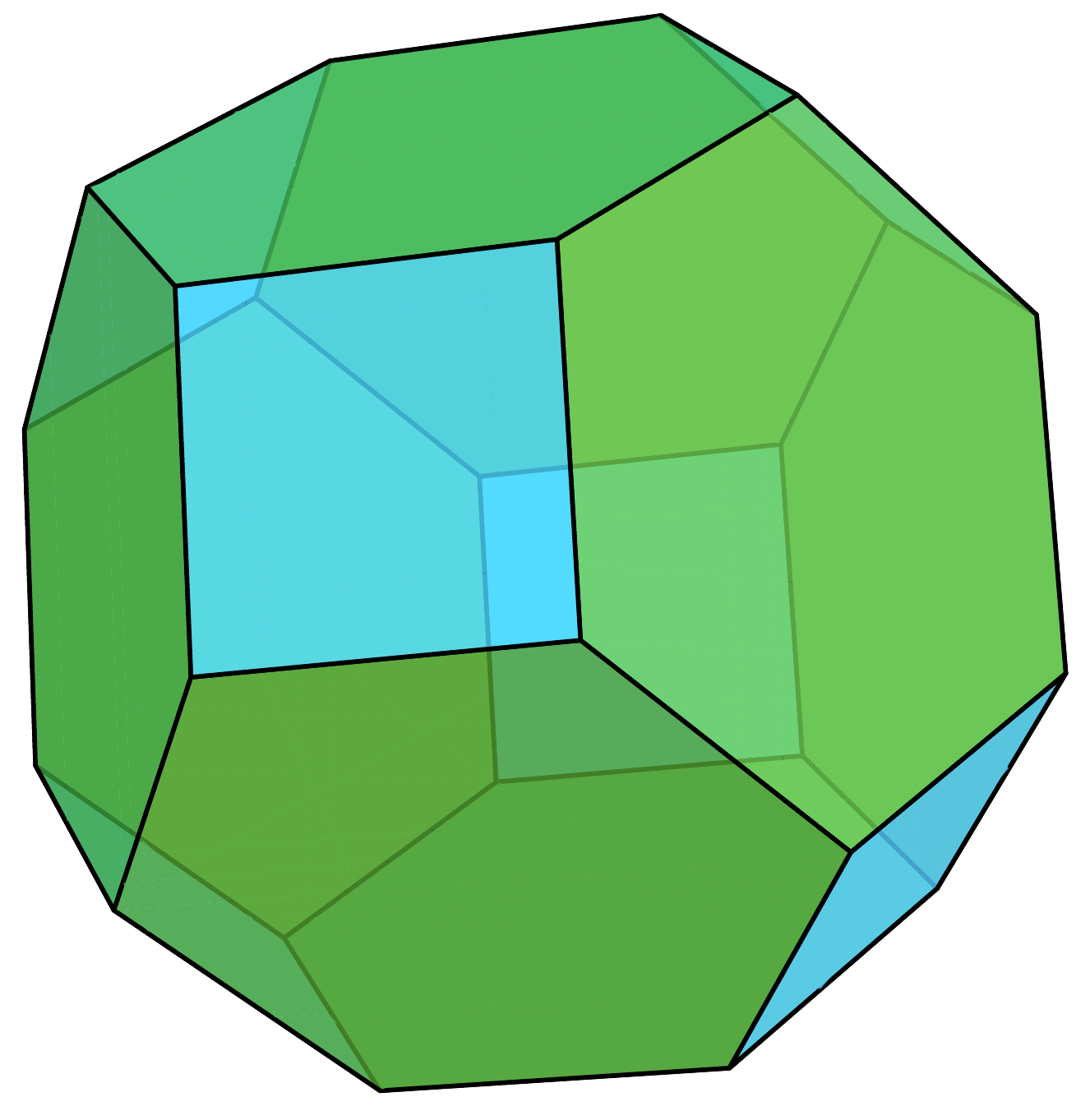}
\caption{The $3$ dimensional permutohedron. This is the graph associahedron for the complete graph on $4$ vertices.}
\end{figure}

\subsection{Context and motivation} In~\cite{Kap93}, Kapranov gives the following beautiful description of the moduli space $\overline M_{0,n}$ as a blowup of $\PP^{n-3}$.

\begin{theorem}[Kapranov]
The moduli space $\overline M_{0,n}$ is isomorphic to the iterated blowup of $\PP^{n-3}$ at $(n-1)$ points $p_1,\ldots, p_{n-1}$ in linear general position, followed by the blowups of the strict transforms of the linear subspaces through these points, in increasing order of dimension.
\end{theorem}

This blowup is manifestly not toric. The maximal \textit{toric} blowup of $\PP^{n-3}$, at the coordinate subspaces in order of increasing dimension, is isomorphic to the toric variety associated to the permutohedron. This gives rise to the alternative modular compactification of $M_{0,n}$ studied by Losev and Manin in~\cite{LM}. Another point of view, taken by Hassett, is to give the marked points ``weights'', allowing markings to collide if their total weight is sufficiently small. The Losev--Manin space is obtained by giving the first two marked points weight $1$, and the remaining points weight $\epsilon$ for sufficiently small $\epsilon$.

Hassett also points out that $\PP^{n-3}$ is itself a modular compactification of $M_{0,n}$, given by the weight vector that assigns weight $1$ to the first mark and weight $\epsilon$ to the remaining marks, for $\epsilon$ sufficiently small. Hassett's perspective yields a large class of interesting compactifications of $M_{0,n}$ lying in between $\PP^{n-3}$ and $\overline M_{0,n}$.

On the other hand, given any finite graph $G$ on $n-2$ vertices, Carr and Devadoss~\cite{CD06} exhibit the graph associahedron $\PG$ as a truncation of the simplex on $n-2$ vertices. This naturally gives rise to a toric blowup of $\PP^{n-3}$. The complete graph $K_{n-2}$ produces the permutohedral variety.

Theorem \ref{thm: mainthm} tells us that there is remarkably little overlap between these two constructions.  Much of the work on the birational geometry of $M_{0,n}$ has focused on modular compactifications such as Hassett's \cite{GJM, Has03, Smyth}.  In a strict sense, the maps between such compactifications are well behaved (see \cite[Theorem 1.15]{Smyth}), leading some to believe that $\overline{M}_{0,n}$ has good Mori-theoretic properties.  The recent proof that $\overline{M}_{0,n}$ is not a Mori Dream Space for $n$ sufficiently large~\cite{CT13, GK14} relies instead on the combinatorics of toric compactifications, suggesting a significant difference between these two points of view.

\[
\begin{tikzcd}
 \phantom{1} & {\color{gray}\overline{M}_{0,n}} \arrow[color=gray]{dr} \arrow[color=gray]{dl} & \phantom{1}\\
\overline{M}_{0,n}^{LM}\arrow{d} \arrow{rr}{\cong} & & X(\mathcal{P}K_{n-2})\arrow{d}\\
\vdots \arrow{d} & & \vdots\arrow{d} \\
\overline{M}_{0,\omega} \arrow[dashed]{rr}{?} \arrow[color=gray]{dr} & & X(\PG) \arrow[color=gray]{dl} \\
 \phantom{1} & {\color{gray}\PP^{n-3}} & \phantom{1}\\
\end{tikzcd}
\]

It is worth noting that, although we are primarily concerned with graph associahedra of connected graphs, the disconnected case has been considered by Carr, Devadoss, and Forcey in~\cite{CDF}. In this formulation, the graph associahedron of the discrete set on $n-2$ vertices is the $(n-3)$-simplex. The associated toric variety is simply $\PP^{n-3}$. In this sense, $\PP^{n-3}$ is a second example, after Losev--Manin, of a toric graph associahedron that is a Hassett space.

Graph associahedra encompass a large number of important polytopes that arise in a multitude of situations in geometry and topology. For instance, the real locus $\overline M_{0,n}(\RR)$ of the Grothendieck--Knudson space of $n$-pointed rational curves has an intrinsic tiling by associahedra~\cite{Sta63} (i.e. the graph associahedron of the path graph).  More generally, the wonderful compactification of a hyperplane arrangement of a Coxeter system has a tiling by graph associahedra~\cite{CD06}.  The cyclohedra first appeared in work of Bott and Taubes on knot invariants~\cite{BT94}. Recently, Bloom has exhibited beautiful connections between the combinatorics of graph associahedra and Floer homology theories~\cite{Bloom11}.

In algebraic geometry, the permutohedron is closely related to the geometry of the Cremona transformation on projective space, and in turn to the closed topological vertex in Gromov--Witten theory~\cite{BK}. The Gromov--Witten theory of toric graph associahedra is considered in~\cite{KRRW}. In the present work, we further explore the connection first noticed by Kapranov, and Losev and Manin, between compactifications of $M_{0,n}$ and the permutohedron.

\subsection*{Acknowledgements} This work was completed as part of the 2014 Summer Undergraduate Mathematics Research at Yale (SUMRY) program, where the first author was a participant and the second and third authors were mentors. We are grateful to all involved in the SUMRY program for the vibrant research community that they helped create. It is a pleasure to thank Dagan Karp, who actively collaborated with the third when the ideas in the present text were at their early stages. We thank Satyan Devadoss for his encouragement, as well as permission to include Figure~\ref{fig: star-4} from~\cite{CDF}. Finally, we thank the referee for their careful reading and comments. The authors were supported by NSF grant CAREER DMS-1149054 (PI: Sam Payne). 

\section{Background}\label{sec: background}

\subsection{Toric graph associahedra} Let $G$ be a connected finite simple graph with vertex set $V(G)=\{0,{\ldots},d \}$. The graph associahedron $\PG$ of $G$ is defined by iterated truncations of the $d$-simplex. Let $\{H_i\}$ denote the set of facets of the $d$-simplex $\Delta_d$, and fix a bijection $V(G)\leftrightarrow \{H_i\}$. Notice that, for every subset of $S\subset V(G)$ there corresponds a unique face of $\Delta_d$, which is the intersection of the facets $H_i$ for $i\in S$.

\begin{definition}
A \textit{tube} in $G$ is a subset $T\subset V(G)$, such that the induced subgraph on the vertices $T$ is connected.  We say that a tube is \textit{trivial} if $\vert T \vert = 1$.  We call a subset of vertices $D$ a \textit{non-tube} if the induced subgraph is not connected.
\end{definition}

The polytope $\PG$ is constructed by the following procedure.

\begin{construction}
\textnormal{Fix a connected graph $G$ on $d+1$ vertices and a bijection between the vertices of $G$ and the facets of $\Delta_d$. Let $T_1,\ldots, T_\ell$ denote the tubes in $G$ having cardinality $d$. Let $f_1,\ldots, f_\ell$ denote the corresponding faces (vertices) of $\Delta_d$. Truncate the faces $f_i$, to produce a polytope $\mathcal P G^{(1)}$. Now, let $T'_1,\ldots, T'_r$ be the tubes in $G$ having cardinality $d-1$. These correspond to faces of dimension $1$ (edges) $e_1,\ldots, e_r$ in $\Delta_d$. Each such edge $e_k$ corresponds to a unique edge in $\mathcal P G^{(1)}$, which we continue to denote $e_k$. Truncate each such edge to obtain a polytope $\mathcal P G^{(2)}$. Proceed inductively, until the cardinality of the tubes are $2$. The resulting iterated truncation is the graph associahedron, denoted $\mathcal PG$. }
\end{construction}

\begin{figure}[h!]
\includegraphics{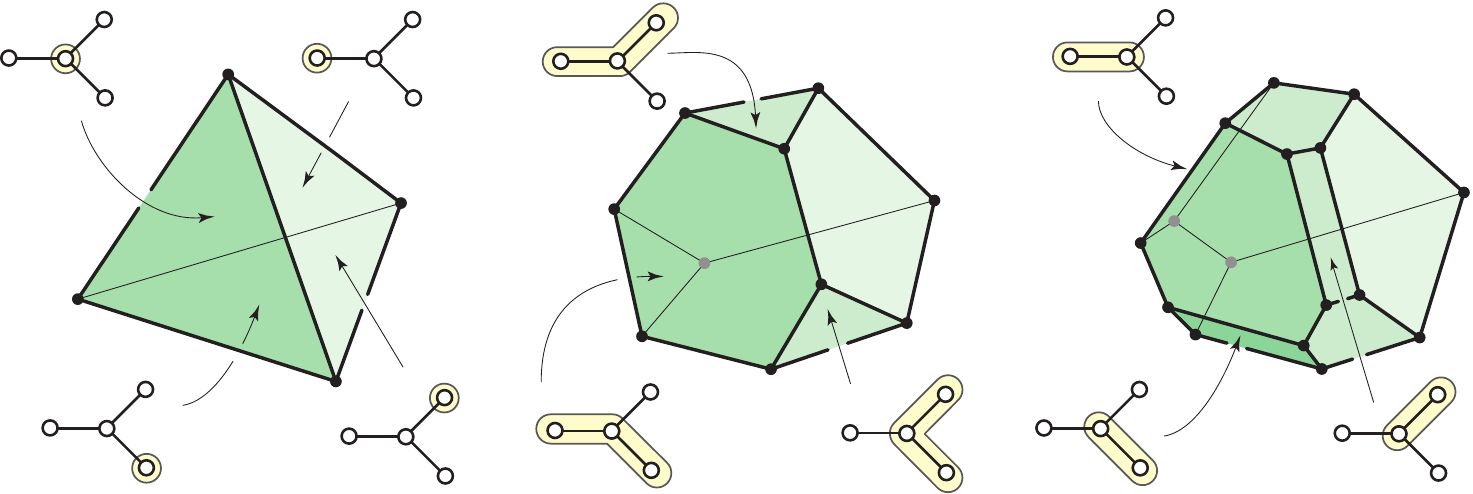}
\caption{The construction of the stellahedron as a truncation of the $3$-simplex. This is the graph associahedron of the $4$-star graph.}
\label{fig: star-4}
\end{figure}

We refer to~\cite{Dev09} for a more thorough description of these truncations, and other aspects of graph associahedra.

A toric variety is obtained from a fan, or a \textit{lattice} polytope, and the above construction does not give $\mathcal PG$ a canonical integer realization. Moreover, there can be distinct integer polytopes, producing different toric varieties, having identical posets of faces. Our point of view is to take the polytope of the toric variety $\PP^d$ (which is combinatorially a simplex) as a starting point, interpreting the construction of Devadoss and Carr as prescribing an iterated blowup of $\PP^d$.  We will refer to these varieties as \textit{toric graph associahedra}, and will denote them $X(\PG)$.

Fix a finite simple graph $G$ on $d+1$ vertices, and let $\Sigma$ denote the fan of the toric variety $\PP^d$. The cones of $\Sigma$ are in natural bijection with the faces of $\Delta_d$. As above, a tube $T$ corresponds to a unique cone $\sigma_T$ of $\Sigma$, and hence a unique coordinate subspace of codimension $|T|$.

\begin{definition}
The graph associahedral fan $\Sigma_G$ is obtained from $\Sigma$ by iterated stellar subdivison along the cones $\sigma_T$ for tubes $T$ of $G$, in increasing order of codimension.
\end{definition}

This fan is independent of the chosen order of subdivision among cones of a given dimension. This follows immediately from~\cite[Theorem 2.6]{CD06}.

\begin{definition}
The toric graph associahedron $X(\PG) := X(\Sigma_G)$, is defined to be be the toric variety associated to the fan $\Sigma_G$. The variety $X(\PG)$ is the iterated blowup of $\PP^d$ along the coordinate subspaces corresponding to tubes $T$ in order of increasing dimension.
\end{definition}

\begin{remark}
We can easily recover a polytope $\PG$ from the toric variety $X(\PG)$. Choose an equivariant projective embedding of $X(\PG)$. Then the poset of faces of the associated lattice polytope can be identified with that of the graph associahedron $\PG$. Alternatively, the canonically associated compactified fan of $\Sigma_G$, as described in~\cite[Section 2]{ACP}, is a polytope whose face poset is identified with that of $\PG$. This also coincides with the Kajiwara--Payne extended tropicalization of the toric variety $X(\PG)$~\cite[Remark 3.3]{Pay09}.
\end{remark}

\subsection{Kapranov's model and Hassett spaces}~\label{sec: Hassett}

The connection to moduli spaces comes from Kapranov's blowup model of $\overline{M}_{0,n}$.  Given a general point $p \in \PP^{n-3}$, there exists a unique rational normal curve $C$ through $p$, the point $p_0 = (1, \ldots , 1)$, and the $n-2$ coordinate points $p_1 , \ldots , p_{n-2}$.  The curve $C$, together with the $n$ points $p, p_0 , \ldots , p_{n-2}$, determines a point in $M_{0,n}$, and in this way we obtain a birational map $\PP^{n-3} \dashrightarrow \overline{M}_{0,n}$.  The indeterminacy loci of this map are the linear spans of subsets of the points $p_i$, and Kapranov~\cite{Kap93} shows that $\overline{M}_{0,n}$ is isomorphic to the blowup of $\PP^{n-3}$ along these linear spans, in increasing order of dimension.

By blowing up the projective space $\PP^{n-3}$ along some subset of these linear spans, one obtains an alternate compactification of $M_{0,n}$.  For example, blowing up $\PP^{n-3}$ along the linear spans of the coordinate points produces the well-known Losev-Manin space $\overline{M}_{0,n}^{LM}$.

Both the Grothendieck--Knudson compactification $\overline{M}_{0,n}$ and the Losev-Manin space $\overline{M}_{0,n}^{LM}$ are examples of a more general construction due to Hassett.  For each weight vector $\omega = (c_M , c_0 , \ldots , c_{n-2} )$ such that $0 < c_i \leq 1$ and $\sum c_i > 2$, Hassett constructs a smooth moduli space of $\omega$-stable curves $\overline{M}_{0,\omega}$ \cite{Has03}.

\begin{definition}
A genus 0 marked curve $(C, p_M , p_0 , \ldots , p_{n-2})$ is $\omega$-stable if
\begin{enumerate}[(S1)]
\item  the only singularities of $C$ are nodes,
\item  the marked points are smooth points of $C$,
\item  the total weight of coincident points is at most 1, and
\item  the line bundle $\omega_C (\sum c_i p_i )$ is ample.
\end{enumerate}
\end{definition}

The last condition can also be rephrased as saying that the total weight of marked points on any component of $C$, plus the number of nodes, must be strictly greater than 2.

We note the following property of Hassett's weighted spaces, which will be useful in the next section of the paper.

\begin{proposition} \cite[Theorem 4.1]{Has03}
Let $\omega = (c_M , c_0 , \ldots , c_{n-2})$ and $\omega' = (c_M' , c_0' , \ldots , c_{n-2}')$ be collections of weight data such that $c_i \geq c_i'$ for all $i\in \{0,\ldots,n-2\}$.  Then there exists a natural birational reduction morphism
$$\rho: \overline{M}_{0,\omega} \to \overline{M}_{0,\omega'} . $$
\end{proposition}

\section{Main results}
\label{sec: main-results}

Throughout this section, $G$ will be a graph on $n-2$ vertices, labeled $v_1,\ldots, v_{n-2}$. Furthermore, we fix a bijection between the set of vertices $\{v_i\}$ and the facets of the $(n-2)$-simplex $\Delta_{n-2}$. We label the markings on an $n$-pointed rational curve by $p_M, p_0, p_1,\ldots, p_{n-2}$. We think of $p_M$ as being the moving point in Kapranov's construction, and $p_0$ as being the identity of the torus in $\PP^{n-3}$.

We begin with the following proposition.

\begin{proposition}
\label{prop:weightrelations}
Let $\omega = (c_M, c_0,\ldots, c_{n-2})$ be a weight vector such that
\[
\overline M_{0,\omega} \cong X(\mathcal P G).
\]
Then we have the following relationships among the entries of $\omega$.
\begin{enumerate}[(W1)]
\item For every nontrivial tube $T\subset V(G)$,
\[
c_0+\sum_{i\in T} c_i>1.
\]
\item For every non-tube $D\subset V(G)$,
\[
c_0+\sum_{j\in D} c_j \leq 1.
\]
\end{enumerate}
\end{proposition}

\begin{proof}
Let $G$ be a graph on $n-2$ vertices, and fix a bijection of the vertices $\{v_i\}$ with the coordinate hyperplanes $\{H_i\}$ of $\PP^{n-3}$. Moreover, we fix an identification of $\overline M_{0,n}$ with the iterated blowup of $\PP^{n-3}$. We consider the reduction map
\[
\rho: \overline M_{0,n}\to \overline M_{0,\omega} = X(\PG).
\]
Let $T$ be a nontrivial tube consisting of vertices $v_{i_1},\ldots, v_{i_e}$, and let $E_T$ be the exceptional divisor in $\overline M_{0,n}$ above the linear subspace $H_{i_1}\cap\cdots \cap H_{i_e}$ in $\PP^{n-3}$. Observe that $\rho$ restricts to an isomorphism on this locus, and thus, the universal curve $\cC_{0,n}$ is $\omega$-stable on $E_T$. Over the generic point of $E_T$, $\cC_{0,n}$ is an $\omega$-stable curve with two components. The components are marked by the sets $I$ and $I^C$, where $I = \{0,i\}_{i\in T}$, whence (W1) follows. Let $D$ be a non-tube and $E_D$ be the exceptional divisor in $\overline M_{0,n}$ above the linear subspace of $\PP^{n-3}$ given by $\bigcap_{i\in D} H_i$. The reduction morphism $\rho$ contracts $E_D$, specifically by forgetting the moduli of the component marked by the set $I$, and similarly to the above case, we conclude (W2). 
\end{proof}

The above proposition yields two natural obstructions for a toric graph associahedron to be a Hassett space.

\noindent \textbf{Obstruction A.} Let $D$ be a non-tube. If there exists a nontrivial tube $T_D\subset D$, then $X(\mathcal P G)$ cannot be isomorphic to $\overline M_{0,\omega}$ for any weight vector $\omega$.

\begin{proof}
Observe that since $T_D$ is a tube, we may apply the inequality (W2) above to obtain
\[
c_0+\sum_{i\in T_D} c_i>1.
\]
On the other hand, $D$ is not a tube, so
\[
c_0+\sum_{j\in D} c_j\leq 1.
\]

Subtracting these inequalities, we obtain
\[
\sum_{i\in D\setminus T_D} c_i <0,
\]
which is impossible.
\end{proof}

\noindent \textbf{Obstruction B.} Suppose there exists a set of vertices $S\subset V(G)$ such that $S$ can be partitioned into $k$ nontrivial tubes and can also be partitioned into $k'$ non-tubes, with $k'\leq k$. Then $X(\mathcal P G)$ cannot be isomorphic to $\overline M_{0,\omega}$ for any weight vector $\omega$.

\begin{proof}
Let
\[
S = \coprod_{i=1}^k T_i = \coprod_{i=1}^{k'} D_i
\]
where the $T_i$'s are nontrivial tubes and the $D_i$'s are non-tubes.  We then have
\begin{eqnarray*}
\sum_{i=1}^k (c_0 + \sum_{j \in T_i} c_j ) = k c_0 +\sum_{i\in S} c_i &>&k\\
\sum_{i=1}^{k'} (c_0 + \sum_{j \in D_i} c_j ) = k' c_0+\sum_{i\in S} c_i&\leq& k'.
\end{eqnarray*}
Subtracting the inequalities, we see that if $k'\leq k$, then $c_0>1$, which is impossible, and we obtain Obstruction B.

\end{proof}

\begin{remark}
Obstruction A is more generally an obstruction to $X(\mathcal P G)$ being modular in the sense of Smyth~\cite{Smyth}. However, Obstruction B is not (see, for example, the discussion in~\cite[Section 7.5]{GJM}).  We note that the class of graphs that are unobstructed by A is strictly larger than that of graphs that are unobstructed by both A and B.  For example, complete bipartite graphs are obstructed by B but not A.  It would be interesting to explore which graph associahedra are isomorphic to modular compactifications of $M_{0,n}$.
\end{remark}

\subsection{Proof of main theorem} Assume that $X(\mathcal P G)$ is isomorphic to a Hassett space, and let $I\subset V(G)$ be a maximal independent set. Suppose there exists a vertex $v$ with $v\notin I$.  By definition, there exists a vertex $w \in I$ such that there is an edge between $v$ and $w$.  If there is another vertex $u \in I$ with no edge connecting it to $v$, then $\{ v,w \}$ is a tube and $\{ u,v,w \}$ is a non-tube, contradicting Obstruction A.  It follows that there must exist edges between $v$ and $u$ for every vertex $u\in I$.

Now, consider another vertex $v' \neq v$ not lying in the independent set $I$.  If $\vert I \vert = 1$, then there is an edge between $v$ and $v'$ by the assumption that $I$ is maximal.  Otherwise, if there is no edge between $v$ and $v'$, then for some $u,u' \in I$, both $\{ u,v \}$ and $\{ u',v' \}$ are tubes, but $\{ u,u' \}$ and $\{ v,v' \}$ are non-tubes, contradicting Obstruction B.  It follows that there is an edge between any two vertices not contained in the independent set $I$.  Inductively, we may reconstruct $G$ from $I$ as an iterated cone.

It remains to show that for such $G$, $X(\mathcal P G)$ is indeed a Hassett space. We write $G$ as an iterated cone
\[
Cone^{n-2-k}\left(\sqcup_{i=1}^k v_i \right).
\]
Recall that we have fixed a bijection between vertices $\{v_i\}$ of $G$, and facets $\{F_i\}$ of $\Delta_{n-3}$. Accordingly, there is a bijection between $\{v_i\}$ and torus invariant points $\{p_1,\ldots, p_{n-2}\}$ of $\PP^{n-3}$, via the induced bijection between $\{v_i\}$ and vertices opposite to the facets $\{F_i\}$.

We call the torus fixed points $\{p_1,\ldots, p_k\}$ corresponding to the discrete base set of $G$ \textit{independent points}, and the remaining points $\{p_{k+1},\ldots, p_{n-2}\}$ \textit{cone points}. Now, let $\omega = (c_M,c_0,c_1,\ldots, c_k,c_{k+1},\ldots, c_{n-2})$.  By symmetry, we may assume that the weight vector is unchanged by permuting the cone points and independent points amongst themselves. We let $c_c$ denote the weight of the cone points and $c_i$ denote the weight of the independent points. The point $p_M$ in Kapranov's construction is not allowed to coincide any of the other marked points, so we are forced to set $c_M = 1$. 

Recall that points are allowed to collide precisely when the sum of their weights is less than $1$. Every cone vertex is connected to every other vertex in the graph, so the non-tubes of $G$ are precisely subsets of independent vertices of size at least two. Thus we may conclude that any collection of independent points may coincide with $p_0$. Conversely, any subset of vertices that contains a cone vertex is a tube, so none of the cone points may collide with $p_0$.  Set $c_i$ equal to $\epsilon$ for $\epsilon$ sufficiently small, as this allows them to collide arbitrarily. We set $c_c = (k+2)\epsilon$ where $k$ is the number of independent points. Finally, we set $c_0 = 1-(k+1)\epsilon$. With these weights $c_0$ cannot collide with the cone points, as the sum of weights exceeds $1$. The space $\overline M_{0,\omega}$ can now be obtained via Hassett's reduction map
\[
\overline M_{0,n}\to \overline M_{0,\omega},
\]
by contracting unstable loci in $\overline M_{0,n}$. By blowing down from Kapranov's realization of $\overline M_{0,n}$, the same arguments above, applied to the universal curve $\mathcal C_{0,n}$ show that this is the same variety $X(\mathcal{P}G)$ obtained by blowup of $\PP^{n-3}$. \qed

\begin{remark}\label{rem: specific-weights}{\bf (Specific weights for the moduli spaces)} The parameter space for weights $(0,1]^n$ has a natural wall and chamber structure, so for $\omega_1,\omega_2$ in the same chamber, the moduli functor for $\omega_1$- and $\omega_2$-stable curves coincide. We record here the weights obtained in the proof above so the reader may have access to it easily.

Let $G$ be a graph on $n-2$ vertices, obtained as an iterated cone over the discrete set on $k$ vertices $v_1,\ldots, v_k$. Then, $X(\mathcal PG)\cong \overline M_{0,\omega}$ for the weight vector $(c_M,c_0,c_1,\ldots, c_{n-2})$ for $\epsilon$ sufficiently small:
\begin{eqnarray*}
c_M &=& 1\\
c_0 &=& 1-(k+1)\epsilon\\
c_c &=& (k+2)\epsilon \ \ \textnormal{when $p_c$ is a cone point} \\
c_i &=& \epsilon \ \ \textnormal{when $p_i$ is an independent point.} \\
\end{eqnarray*}
\end{remark}

%
%
%

\section{Examples}

\subsection{Graphs on $3$ vertices}

There are two connected graphs on $3$ vertices, the cycle $K_3$, and the path graph $P_3$. The toric graph associahedron $X(\mathcal P K_3)$ is the blowup of $\PP^2$ at its $3$ torus fixed points. The toric graph associahedron $X(\mathcal P P_3)$ is the blowup of $\PP^2$ at $2$ torus fixed points. The Grothendieck--Knudson compactification of the moduli space of $5$-pointed curves is isomorphic to the blowup of $\PP^2$ at its $3$ torus fixed points, and the identity of the torus.

The cycle on $3$ vertices coincides with the complete graph, and thus the toric graph associahedron $X(\mathcal P K_3)$ is isomorphic to the Losev--Manin compactification $\overline M_{0,\omega}$ for $\omega = (1,1,\epsilon,\epsilon,\epsilon)$. On the other hand, the path graph $P_3$ can be viewed as the cone over the discrete set on $2$ vertices. The toric graph associahedron $X(\mathcal P P_3)$ is isomorphic to $\overline M_{0,\omega'}$ for $\omega' = (1,\frac{1}{2},\frac{1+\epsilon}{2},\epsilon,\epsilon)$.

The Losev--Manin stable $5$-pointed rational curves are those chains of $\PP^1$'s, where $p_1$ and $p_2$ lie on either end of the chain, and each component has at least one of the light points $p_3,p_4$, or $p_5$. Thus, the chain can have at most $3$ components.

If $C$ is a $\omega'$-stable $5$-pointed rational curve with two components, then $p_2$ and $p_3$ must lie on the same component as each other, but a different component than $p_1$. Moreover, the light points $p_4$ and $p_5$ must lie on distinct components of $C$. In fact, $C$ cannot have $3$ components. To see this, assume otherwise. Then $C$ is a chain of $\PP^1$'s
\[
C = C_1\cup C_2\cup C_3.
\]
We may assume without loss of generality that $p_1$ lies on $C_1$. Then $p_2$ and $p_3$ must both lie on $C_3$. To stabilize the components $C_1$ and $C_3$, they must both be additionally marked by one of the light points. However, now the central component $C_2$ has $2$ nodes and no marks, and is unstable.

\[
\begin{tikzcd}
\phantom{1} & {\color{gray}\overline{M}_{0,5} \cong Bl_{4}(\PP^2)} \arrow[color=gray]{dr} \arrow[color=gray]{dl} & \phantom{1}\\
\overline{M}_{0,n}^{LM}\arrow{d} \arrow{rr}{\cong} & & X(\mathcal{P}{K_3})\arrow{d}\\
\overline{M}_{0,\omega'} \arrow[swap]{rr}{\cong} \arrow[color=gray]{dr} & & X(\mathcal P P_3) \arrow[color=gray]{dl} \\
\phantom{1} & {\color{gray}\PP^{2}} & \phantom{1}\\
\end{tikzcd}
\]

Note that, although $X (\mathcal P P_3)$ is isomorphic to $\overline{M}_{0,\omega'}$, the universal families over the two spaces are different.  For example, in the toric model, if the moving point lies on the line between one of the points of weight $\frac{1}{2}$ and one of the points of weight $\epsilon$, then the unique conic through these 5 points is a union of two lines, and the resulting pointed curve is not $\omega'$-stable.

\subsection{Graphs on $4$ vertices} There are, up to symmetry, $6$ different graphs on $4$ vertices, but only $3$ graphs that can be obtained as iterated cones over discrete sets. These graphs are the complete graph $K_4$, the complete graph minus a single edge $V_4$, and the star graph on $4$ vertices $S_4$.

\begin{figure}[h!]
\begin{tikzpicture}[scale=1.25]
\draw [ball color=black] (0,0) circle (0.5mm);
\draw [ball color=black] (1,0) circle (0.5mm);	
\draw [ball color=black] (1,1) circle (0.5mm);
\draw [ball color=black] (0,1) circle (0.5mm);

\draw (0,1)--(1,1)--(1,0); \draw (1,1)--(0,0);

\draw [ball color=black] (3,0) circle (0.5mm);
\draw [ball color=black] (4,0) circle (0.5mm);	
\draw [ball color=black] (4,1) circle (0.5mm);
\draw [ball color=black] (3,1) circle (0.5mm);

\draw (3,0)--(3,1)--(4,1)--(4,0)--(3,0)--(4,1);

\draw [ball color=black] (6,0) circle (0.5mm);
\draw [ball color=black] (7,0) circle (0.5mm);	
\draw [ball color=black] (7,1) circle (0.5mm);
\draw [ball color=black] (6,1) circle (0.5mm);

\draw (6,0)--(6,1)--(7,1)--(7,0)--(6,0)--(7,1); \draw (7,0)--(6,1);

\end{tikzpicture}
\caption{Graphs on $4$ vertices giving rise to Hassett spaces.}
\end{figure}
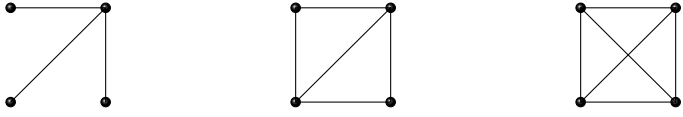

We discuss the star graphs and complete graphs in greater generality in the forthcoming subsection.

The graph $V_4$ can be obtained as the twice iterated cone over the discrete set on $2$ vertices, $Cone^2(v_1\sqcup v_2)$. The toric variety $X(\mathcal P V_4)$ is isomorphic to $\overline M_{0,\omega}$ for
\[
\omega = (1,1-3\epsilon, 4\epsilon,4\epsilon,\epsilon, \epsilon).
\]

This space can be obtained from the Losev-Manin space $\overline{M}_{0,6}^{LM}$ by blowing down the divisor $\Delta_{134}$.  To put this another way, suppose that $C = C_1 \cup C_2$ is a curve with two components, with three marked points on $C_1$ and three marked points on $C_2$.  Then $C$ is $\omega$-stable if and only if at least one of the ``light'' points with weight $\epsilon$ lies on the same component as the ``heavy'' point with weight 1.

\subsection{Complete graphs and star graphs} As we have discussed above, when $G = K_{n-2}$, then $X(\mathcal P K_{n-2})$ is the permutohedral variety, isomorphic to the Losev--Manin compactification of $M_{0,n}$. The graph $K_{n-2}$ is obtained as the $(n-3)$-times iterated cone over a single vertex. On the other extreme, we may consider the cone over the discrete set on $n-3$ vertices. The resulting graph is the star graph on $n-2$ vertices, denoted $S_{n-2}$.

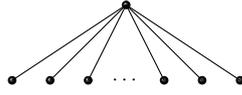
\begin{figure}[h!]
\begin{tikzpicture}
\draw [ball color = black] (-1.5,0) circle (0.5mm);
\draw [ball color = black] (-1,0) circle (0.5mm);
\draw [ball color = black] (-0.5,0) circle (0.5mm);

\draw [ball color = black] (1.5,0) circle (0.5mm);
\draw [ball color = black] (1,0) circle (0.5mm);
\draw [ball color = black] (0.5,0) circle (0.5mm);
\draw node at (0,0) {\tiny $\cdots$};

\draw [ball color=black] (0,1) circle (0.5mm);

\foreach \a in {-1.5,-1,-0.5,1.5,1,0.5}
	\draw (0,1)--(\a,0);
\end{tikzpicture}
\caption{The star graph is a cone over a discrete set.}
\end{figure}

Observe that the tubes of $S_d$ are single vertices, or any subset of the vertices containing the cone point. Using this fact, it is straightforward to check that the graph associahedron $\mathcal P S_d$ can be described as follows. Choose a distinguished facet $F_0$ of the simplex $\Delta_d$, corresponding to the unique high-valence vertex of $S_d$. Then, $\mathcal P S_d$ may be obtained from $\Delta_d$ by truncating all faces lying in $F_0$. Correspondingly, the toric variety $X(\mathcal P S_d)$ is obtained by choosing a distinguished coordinate hyperplane $H_0$, and blowing up all coordinate planes contained in $H_0$. Observe that the proper transform $E$ of $H_0$ in $X(\mathcal P S_d)$ is isomorphic to the $(d-1)$ dimensional permutohedral variety.

It follows from the main result that
\[
X(\mathcal P S_{n-2}) \cong \overline M_{0,\omega},
\]
where $\omega = (1,\frac{1}{2},\frac{1}{2}+\epsilon,\epsilon,\epsilon,\ldots, \epsilon)$.

The locus of curves consisting of two components, where one component is marked with $p_2$ and $p_3$ is isomorphic to the Losev--Manin compactification of $M_{0,{n-1}}$. It is straightforward to check that this stratum coincides with the torus invariant exceptional divisor $E$ of $X(\mathcal P S_{n-2})$ described above. 

\begin{remark}
For different choices of weights $\omega_1$ and $\omega_2$, it may be possible for $\overline M_{0,\omega_1}$ to be isomorphic to $\overline M_{0,\omega_2}$ as varieties but \textit{not} as moduli spaces, i.e. the universal families may not coincide. For instance, choosing $\omega = (1,\frac{1}{2},\frac{1}{2},\epsilon,\epsilon,\ldots, \epsilon)$, the space $\overline M_{0,\omega}$ is isomorphic to $X(\mathcal P S_{n-2})$ as above. Here, the Losev--Manin compactification of $M_{0,n-1}$ appears as the locus where $p_2$ and $p_3$ coincide.
\end{remark}

\bibliographystyle{siam}
\bibliography{ToricGraphAssociahedra}

\begin{thebibliography}{10}

\bibitem{ACP}
{\sc D.~Abramovich, L.~Caporaso, and S.~Payne}, {\em The tropicalization of the
  moduli space of curves}, Ann. Sci. {\'E}c. Norm. Sup{\'e}r.,  (To appear).

\bibitem{Bloom11}
{\sc J.~M. Bloom}, {\em A link surgery spectral sequence in monopole floer
  homology}, Adv. Math., 226 (2011), pp.~3216--3281.

\bibitem{BT94}
{\sc R.~Bott and C.~Taubes}, {\em On the self-linking of knots}, J. Math.
  Phys., 35 (1994), pp.~5247--5287.

\bibitem{BK}
{\sc J.~Bryan and D.~Karp}, {\em The closed topological vertex via the
  {C}remona transform}, J. Algebraic Geom., 14 (2005), pp.~529--542.

\bibitem{CDF}
{\sc M.~Carr, S.~L. Devadoss, and S.~Forcey}, {\em Pseudograph associahedra},
  Journal of Combinatorial Theory, Series A, 118 (2011), pp.~2035--2055.

\bibitem{CD06}
{\sc M.~P. Carr and S.~L. Devadoss}, {\em Coxeter complexes and
  graph-associahedra}, Topology Appl., 153 (2006), pp.~2155--2168.

\bibitem{CT13}
{\sc A.-M. Castravet, J.~Tevelev, et~al.}, {\em $\overline{M}_{0,n}$ is not a
  mori dream space}, Duke Math. J., 164 (2015), pp.~1641--1667.

\bibitem{KRRW}
{\sc {D. Karp, D. Ranganathan, P. Riggins, and U. Whitcher}}, {\em Toric
  symmetry of $\mathbb{CP}^3$}, Adv. Theor. Math. Phys., 4 (2012),
  pp.~1291--1314.

\bibitem{Dev09}
{\sc S.~L. Devadoss}, {\em A realization of graph associahedra}, Discrete
  Mathematics, 309 (2009), pp.~271--276.

\bibitem{GJM}
{\sc N.~Giansiracusa, D.~Jensen, and H.-B. Moon}, {\em {GIT}
  {C}ompactifications of ${M}_{0,n}$ and {F}lips}, Adv. Math.,  (2012).

\bibitem{GK14}
{\sc J.~L. Gonz{\'a}lez and K.~Karu}, {\em Some non-finitely generated cox
  rings}, arXiv:1407.6344,  (2014).

\bibitem{Has03}
{\sc B.~Hassett}, {\em Moduli spaces of weighted pointed stable curves}, Adv.
  Math., 173 (2003), pp.~316--352.

\bibitem{Kap93}
{\sc M.~Kapranov}, {\em {Chow quotients of Grassmannians. {I}}}, in {IM
  Gel′fand Seminar}, vol.~16, 1993, pp.~29--110.

\bibitem{LM}
{\sc A.~Losev and Y.~Manin}, {\em New moduli spaces of pointed curves and
  pencils of flat connections}, Michigan Math. J., 48 (2000), pp.~443--472.

\bibitem{Pay09}
{\sc S.~Payne}, {\em Analytification is the limit of all tropicalizations},
  Math. Res. Lett., 16 (2009), pp.~543--556.

\bibitem{Smyth}
{\sc D.~I. Smyth}, {\em Towards a classification of modular compactifications
  of the moduli space of curves}, Invent. Math., 192 (2013), pp.~459--503.

\bibitem{Sta63}
{\sc J.~D. Stasheff}, {\em Homotopy associativity of {H}-spaces. {I}}, Trans.
  Amer. Math. Soc.,  (1963), pp.~275--292.

\end{thebibliography}

\end{document}